\documentclass[11pt,a4paper]{article}
\usepackage{amsmath,amssymb,latexsym,amsthm,relsize}
\usepackage[english]{babel}
\usepackage[latin2]{inputenc}
\newtheorem{theorem}{Theorem}
\newtheorem{lemma}[theorem]{Lemma}
\newtheorem{proposition}[theorem]{Proposition}
\newtheorem{corollary}[theorem]{Corollary} 

\theoremstyle{definition}
\newtheorem{definition}{Definition}

\newtheorem{construction}{Construction}
\theoremstyle{remark}
\newtheorem{remark}{Remark}

\begin{document} 
	
	\title{Inverse property of non-associative abelian extensions}
	\author{\'Agota Figula (Debrecen) and P\'eter T. Nagy (Budapest)}
	\date{}
	\maketitle 
\maketitle

\begin{abstract}
Our paper deals with the investigation of extensions of commutative groups by loops so that the quasigroups that result in the multiplication between cosets of the kernel subgroup are T-quasigroups. We limit our study to extensions in which the quasigroups determining the multiplication are linear functions without constant term, called linear abelian extensions. 
	We characterize constructively such extensions with left-, right-, or inverse properties using a general construction according to an equivariant group action principle. We show that the obtained constructions can be simplified for ordered loops. Finally, we apply our characterization to determine the possible cardinalities of the component loop of finite linear abelian extensions.
\end{abstract}

\section{Introduction} 

 A. A. Albert and R. H. Bruck proved in 1944 (cf. \cite{Alb}, \cite{Bru}) that the construction of loop extensions of a loop $N$ by a loop $K$ has a large degree of freedom, namely the multiplication between different 
cosets $\neq N$ of $N$ can be given by arbitrary quasigroup multiplication and for multiplication of $N$ with a coset of $N$ or a coset of $N$ with $N$ one has to choose quasigroups with left, respectively right unit elements.\\	
In the following, we want to study extensions of a commutative group $A$ by a loop $L$, so  that the quasigroups, which determine the multiplications betwen cosets of $A$, are so-called {\it T-quasigroups}. The theory of T-quasigroups was created almost 50 years ago by T. Kepka and P. N\v{e}mec (cf. \cite{KeNe1} and \cite{KeNe2}). The multiplication of a T-quasigroup $Q$  over an abelian group $A$ has the form $x\cdot y = \phi(x) + \psi(y) + g$, $x, y\in Q$,  where "$+$" is the  addition of $A$, $\phi$, $\psi$ are automorphisms of $A$ and $g$ is a constant in $A$. A general theory of natural generalizations of T-quasigroups has been extensively developed (e.g. \cite{StVo1}, \cite{Shc}, Chapter 2.10).  \\
The present work deals with the investigation of extensions of commutative groups by loops so that the quasigroups that result in the multiplication between cosets of the kernel subgroup are T-quasigroups. These loop extensions are introduced and examined by D. Stanovsk\'y,\, P. Vojt\v{e}chovsk\'y in \cite{StVo2} under the name {\it abelian extensions}. We limit our study to extensions in which the quasigroups determining the multiplication are linear functions $x\cdot y = \phi(x) + \psi(y)$, $x, y\in Q$, $\phi, \psi \in\hbox{Aut}(A)$ without constant term.\\
We characterize constructively such linear abelian extensions with left inverse, right inverse, or inverse properties using a general construction according to an equivariant group action principle. This type of construction of quasigroups and loops was originally proposed in \cite{SuKr} and applied in \cite{NS} to describe Schreier-type loop extensions with special properties.\\
After an introduction and presentation of the necessary concepts, we examine in \S 3 the property of equality of left and right inverses in the extension.
 \S 4 is devoted to the discussion of the left and right inverse property of linear abelian extensions. In \S 5 we assume that the quasigroups giving the multiplication between the kernel subgroup and its cosets are identical with the kernel subgroup. In this case we prove that an equivariant action of the symmetry group $S_3$ on $L\times L$, respectively on $\hbox{Aut}(A)\times\hbox{Aut}(A)$ is a constructive characterization of the inverse property. We show that the obtained constructions can be simplified if $L$ is an ordered loop. Finally, we apply our characterization to determine the possible cardinalities of the component loops $L$ of finite linear abelian extension loops.

\section{Preliminaries} 

A {\it quasigroup} $L$ is a set with a multiplication map $(x,y)\mapsto x\cdot y:L\times L\to L$ such that for each $x\in L$ the  {\it left translations}  $\lambda_x :L\to L$,  $\lambda_xy = xy$,  and the {\it right translations}  $\rho_x : L\to L$, $\rho_xy =  yx$, are bijective maps. The left and right division operations on $L$ are defined by the maps $(x,y)\mapsto x\backslash y = \lambda_x^{-1}y$, respectively  $(x,y)\mapsto x/y = \rho_y^{-1}x$,  
$x,y\in L$. An element $e\in L$ is called left (right) identity if it satisfies $e\cdot x = x$ ($x\cdot e = x$) for any $x\in L$. A left and right identity is called identity element. A quasigroup $L$ is a {\it loop} if it has an identity element. The automorphism group of $L$ is denoted by $\mathrm{Aut}(L)$. The multiplication $x\star y =  y\cdot x$ on a loop $L$ with multiplication $x\cdot y$ defines the {\it opposite loop} of $L$.\\
We will reduce the use of parentheses by the following convention: juxtaposition will denote multiplication, the division operations are less binding than juxtaposition, and the multiplication symbol is less binding than the divisions. For instance the expression $xy / u \cdot v \backslash w$ is
a short form of $((x\cdot y)/ u)\cdot(v\backslash w)$.\\ 
The {\it left inverse}, respectively the {\it right inverse} of an element $x$ of a loop $L$ is $e/x$, respectively 
$x \backslash e$ since $e/x\cdot x=e$, respectively $x\cdot x \backslash e=e$ holds. If the left and right inverses of $x\in L$ coincide then $x$ has {\it two-sided inverse} denoted by $x^{-1} = e/x = x \backslash e$. \\
A loop $L$ satisfies the {\it left}, respectively the {\it right inverse property} 
if there exists a bijection $\iota:L\to L$, such that $\iota(x)\cdot xy = y$, respectively 
$yx\cdot \iota(x)= y$ holds for all $x, y\in L$.  It is well known that in loops with left or right inverse property the left and right inverses of any element coincide, (cf. 
\cite{Pfl}, I.4.2 Theorem), hence $\iota(x) = x^{-1}$. A loop with {\it left} and {\it right inverse property} has {\it  inverse property}.  \\
A subloop $N\subset L$ is {\it normal} if it is the kernel of a homomorphism of $L$. The {\it factor loop} $L/N$ is the loop induced on the set of left cosets of the normal subloop $N$. A loop $L$ is an extension of a loop $N$ by a loop $K$ if $N$ is a normal subloop of $L$, called the {\it kernel} of the extension, and $K$ is isomorphic to the factor loop $L/N$. \\
An {\it ordered loop} $L$ is a loop together with an order $\le $ on $L$ satisfying the monotonic laws: if $x < y$, then $xz < yz$ and $zx < zy$ for any $x,y,z\in L$, where $x < y$ means $x \le y$ and $x \neq y$. An element $x\in L$ said to be {\it positive} if $e< x$ and {\it negative}, if $x< e$. The monotonic law implies that if $x \in  L$ is positive then $x\backslash e$ and $e/x$ are negative and conversely (cf. \cite{KaPC}). 
\subsection*{Linear abelian extensions} 
Let $A = (A,+)$ be a commutative group and $L = (L,\cdot, /, \backslash )$ a loop with identity element $\epsilon\in L$. A pair $(P, Q)$ is called a \emph{loop cocycle} if $P, Q$ are mappings $L\times L\to\hbox{Aut}(A)$
satisfying  $P(\alpha,\epsilon)=\hbox{Id}=Q(\epsilon,\beta)$ for every $\alpha,\beta \in L$. 
\begin{definition} The \emph{linear abelian extension $F(P,Q)$ of the group $A$ by the loop $L$ determined by the cocycle $(P, Q)$} is defined by the multiplication 
\begin{equation} \label{ext2} (\alpha,a) \cdot (\beta,b)=(\alpha\beta,P(\alpha,\beta) a + Q(\alpha,\beta) b) \end{equation} 
on $L \times A$.  \end{definition}
Clearly, $F(P,Q)$ is a loop with identity $(\epsilon,0)$. We have 
\begin{lemma}
$F(P,Q)$ is commutative if and only if $L$ is commutative and the cocycle $(P, Q)$ satisfies $P(\alpha,\beta)=Q(\beta,\alpha)$ for all $\alpha,\beta \in L$. 
\end{lemma}

\section{Coincidence of the left and right inverses} 

Let $F(P,Q)$ be a linear abelian extension of the group $A$ by a loop $L$ not necessarily having two-sided inverses, we denote by $\epsilon/\xi$ the left inverse and by $\xi\backslash\epsilon$  the right inverse of $\xi\in L$.

\begin{lemma}\label{leftandrightinverses}
	The left, respectively the right inverses of an element $(\xi,x)$ of $F(P,Q)$ have the expressions 
	\begin{equation} \label{lri} \begin{split} (\epsilon,0)/(\xi,x) = (\epsilon/\xi,-P(\epsilon/\xi,\xi)^{-1} Q(\epsilon/\xi,\xi) x),\\ (\xi,x)\backslash(\epsilon,0) = (\xi\backslash\epsilon,-Q(\xi,\xi\backslash\epsilon)^{-1} P(\xi,\xi\backslash\epsilon)x).\end{split}\end{equation} 
\end{lemma} 
\begin{proof} 
If $(\epsilon,0)/(\xi,x) = (\eta,y)$ then $(\eta,y)(\xi,x)=(\eta\xi,P(\eta,\xi)y+Q(\eta,\xi)x)= (\epsilon,0)$. Similarly, if $(\xi,x)\backslash(\epsilon,0) = (\eta,y)$ then 
$(\xi,x)(\eta,y)=(\xi\eta,P(\xi,\eta)x+Q(\xi,\eta)y)=(\epsilon,0)$. Expressing $(\eta,y)$ from these equations we get the assertion.
\end{proof}
\begin{proposition}
The left and right inverses of any element of $F(P,Q)$ coincide if and only if $L$ has this property and 
	\begin{equation} \label{cip} p(\xi^{-1}) = q(\xi^{-1})p(\xi)^{-1}q(\xi) \quad \hbox{holds for all} \quad \xi\in L,  
\end{equation}
where $p: L\to\hbox{Aut}(A)$ and $q: L\to\hbox{Aut}(A)$ are the maps defined by \[p(\xi) = P(\xi^{-1},\xi)\quad \hbox{and} \quad q(\xi) = Q(\xi^{-1},\xi).\]
\end{proposition}	
\begin{proof} 
It follows from Lemma \ref{leftandrightinverses} that the left and right inverses of elements of $F(P,Q)$ coincide if and only if for all $\xi\in L$ 
\[\xi^{-1} = \epsilon/\xi = \xi\backslash\epsilon \quad \hbox{and} \quad -P(\xi^{-1},\xi)^{-1} Q(\xi^{-1},\xi) = -Q(\xi,\xi^{-1})^{-1} P(\xi,\xi^{-1}),\] 
which is equivalent to the assertion. 
\end{proof}
In the following we construct maps $p, q: L\to\hbox{Aut}(A)$ satisfying the condition (\ref{cip}).
\begin{construction} \label{pqcons}  Let be $q: L\to\hbox{Aut}(A)$ with $q(\epsilon) = \hbox{Id}$  an arbitrary map. The elements of the orbits $\{\xi,\xi^{-1}\}$ of the group generated by the map $\xi \mapsto \xi^{-1}$ are interchanged if $\xi\neq \xi^{-1}$. For any orbit with two elements we choose freely the value $p(\xi)$ of $p: L\to\hbox{Aut}(A)$ at one of the elements $\xi\in\{\xi,\xi^{-1}\}$, and define at the other element $\xi^{-1}\in\{\xi,\xi^{-1}\}$ the value $p(\xi^{-1}):= q(\xi^{-1})p(\xi)^{-1}q(\xi)$. Denoting $\eta =\xi^{-1}$ and computing $p(\eta^{-1})$ we obtain  
	\[p(\eta^{-1}) = q(\eta^{-1})p(\eta)^{-1}q(\eta) = q(\xi)p(\xi^{-1})^{-1}q(\xi^{-1}) = p(\xi),\] 
	which means that the map $\iota: \xi \mapsto \xi^{-1}$ induces the involution  $\mathcal{I}:(\xi,p(\xi))\mapsto (\xi^{-1},p(\xi^{-1}))$.
 Hence $p: L\to\hbox{Aut}(A)$ is well defined on the set $\{\xi\in L;\ \xi\neq \xi^{-1}\}$. If $\xi = \xi^{-1}$ we choose the value $p(\xi)$ satisfying $\left(p(\xi)^{-1}q(\xi)\right)^2 = \hbox{Id}$, particularly $p(\epsilon) = \hbox{Id}$. Consequently, the condition (\ref{cip}) is satisfied at all $\xi\in L$. Let be \[P(\xi,\epsilon): = \hbox{Id},\quad Q(\epsilon,\xi) := \hbox{Id},\quad P(\xi^{-1},\xi) := p(\xi),\quad Q(\xi^{-1},\xi) := q(\xi)\] 
	and define $P(\xi,\eta)$ and $Q(\xi,\eta)$ arbitrarily on the complement of the subset \[\Sigma = \{(\xi,\epsilon); \;\xi\in L\}\cup\{(\epsilon,\xi); \;\xi\in L\}\cup\{(\xi^{-1},\xi); \;\xi\in L\}\subset L\times L.\] \end{construction}  
We obtain the following 
\begin{corollary}
The left and right inverses of all elements of  $F(P,Q)$ coincide if and only if $L$ has this property and the equations $P(\xi^{-1},\xi) = p(\xi)$ and $Q(\xi^{-1},\xi) = q(\xi)$ are satisfied for any $\xi\in L$, where the maps $p, q: L\to\hbox{Aut}(A)$ are defined in Construction \ref{pqcons}.
\end{corollary}
\begin{remark} Construction \ref{pqcons} can be simplified if $L$ is an ordered loop. If $\xi\neq\epsilon$ then $\xi\neq \xi^{-1}$, one of the elements of the orbit $\{\xi,\xi^{-1}\}$ is positive and the other is negative. For positive elements of $L$ we choose freely the value $p(\xi)$, and define at $\xi^{-1}$ by $p(\xi^{-1}):= q(\xi^{-1})p(\xi)^{-1}q(\xi)$. Let be $p(\epsilon) = \hbox{Id}$.
Hence $p: L\to\hbox{Aut}(A)$ is well defined. Consequently, the condition (\ref{cip}) is satisfied at all $\xi\in L$.\end{remark}

\section{Extensions with left or right inverse property} \label{leftinv} 

\subsection*{Left inverse property} 
\begin{proposition} \label{leftinverse} 
The extension $F(P,Q)$ has the left inverse property if and only if $L$ has the left inverse property and the equations   
\begin{equation} \label{eq3} Q(\xi^{-1},\xi\eta) = Q(\xi,\eta)^{-1}, \quad P(\xi^{-1},\xi\eta) = Q(\xi,\eta)^{-1}P(\xi,\eta)Q(\xi^{-1},\xi)^{-1}P(\xi^{-1},\xi) \end{equation} 
hold for all $\xi,\eta \in L$.
\end{proposition}
\begin{proof}
$F(P,Q)$ has the left inverse property if and only if 
\[(\xi,x)^{-1} \cdot (\xi,x) (\eta,y) = (\xi^{-1},-P(\xi^{-1},\xi)^{-1} Q(\xi^{-1},\xi) x)(\xi\eta,P(\xi,\eta)x+Q(\xi,\eta)y) =(\eta,y),\] or equivalently $L$ has this property and we have 
\begin{equation} \label{eq1}  -P(\xi^{-1},\xi\eta)P(\xi^{-1},\xi)^{-1} Q(\xi^{-1},\xi)x + Q(\xi^{-1},\xi\eta)(P(\xi,\eta)x+ Q(\xi,\eta)y) = y\end{equation}  
for all $\xi,\eta \in L$ and $x,y \in A$. This is equivalent to the identities
\begin{equation} \label{eq2} Q(\xi^{-1},\xi\eta) Q(\xi,\eta)=\hbox{Id} \quad \hbox{and} \quad P(\xi^{-1},\xi\eta) = Q(\xi^{-1},\xi\eta)P(\xi,\eta)Q(\xi^{-1},\xi)^{-1}P(\xi^{-1},\xi). \end{equation} 
Replacing $Q(\xi^{-1},\xi\eta) = Q(\xi,\eta)^{-1}$ into the second identity we obtain the assertion. 
 \end{proof}
Now, we build linear abelian extensions $F(P,Q)$ satisfying the left inverse property. We use the identities obtained from (\ref{eq3}) putting $\xi^{-1}$ into $\xi$ and $\xi$ into $\eta$:
\[Q(\xi,\epsilon) = Q(\xi^{-1},\xi)^{-1}, \quad P(\xi,\epsilon) = Q(\xi^{-1},\xi)^{-1}P(\xi^{-1},\xi)Q(\xi,\xi^{-1})^{-1}P(\xi,\xi^{-1}).\]
The first of these identities means $Q(\xi,\epsilon) = q(\xi)^{-1}$ and the second one implies (\ref{cip}).
\begin{construction}\label{conlip}  We define the loop cocycle $(P,Q)$ on the subset \[\Sigma = \{(\xi,\epsilon); \;\xi\in L\}\cup\{(\epsilon,\xi); \;\xi\in L\}\cup\{(\xi^{-1},\xi); \;\xi\in L\}\subset L\times L\] as follows:
\begin{equation}\label{sigma} \begin{split} &P(\xi,\epsilon):= \hbox{Id},\; Q(\epsilon,\xi):= \hbox{Id},\\ &\xi\mapsto P(\epsilon,\xi):L\setminus\{\epsilon\}\to\hbox{Aut}(A) \; \hbox{is an arbitrary map},  \\
&Q(\xi^{-1},\xi) = Q(\xi,\epsilon)^{-1} := q(\xi),\;\hbox{where}\; q: L\setminus\{\epsilon\}\to\hbox{Aut}(A) \;\hbox{is an arbitrary map},\\ &P(\xi^{-1},\xi) := p(\xi),\;\hbox{where}\; p:L\to\hbox{Aut}(A)\; \hbox{is satisfying the condition (\ref{cip})}.\end{split}\end{equation}
The maps $p:L\to\hbox{Aut}(A)$ satisfying the equation (\ref{cip}) are described in Construction \ref{pqcons}. The permutation $\varphi:(\xi,\eta)\mapsto (\xi^{-1},\xi\eta)$ acting on the set $(L \times L)\setminus\Sigma$ interchanges the pairs of elements of the orbits of the group $\Gamma_\varphi$ generated by $\varphi$. For any orbit of $\Gamma_\varphi$ we choose arbitrarily the value $Q(\xi,\eta)$, respectively $P(\xi,\eta)$ of the loop cocycle $(P,Q)$ at one of the elements of the orbit, and define the value $Q(\varphi(\xi,\eta)) = Q(\xi^{-1},\xi\eta)$, respectively $P(\varphi(\xi,\eta)) = P(\xi^{-1},\xi\eta)$, at the other element $\varphi(\xi,\eta)$ by 
\begin{equation}\label{minussigma} Q(\xi^{-1},\xi\eta) := Q(\xi,\eta)^{-1}, \quad P(\xi^{-1},\xi\eta) := Q(\xi,\eta)^{-1}P(\xi,\eta)Q(\xi^{-1},\xi)^{-1}P(\xi^{-1},\xi).\end{equation} 
Putting $\varphi(\xi,\eta) = (\xi^{-1},\xi\eta)$ into $(\xi,\eta)$ we obtain from (\ref{minussigma})
the equations
\begin{equation}\label{plussigma} Q(\xi,\eta) = Q(\xi^{-1},\xi\eta)^{-1}, \quad P(\xi,\eta) = Q(\xi^{-1},\xi\eta)^{-1}P(\xi^{-1},\xi\eta)Q(\xi,\xi^{-1})^{-1}P(\xi,\xi^{-1}).\end{equation}
We express from the second equation \[P(\xi,\eta)^{-1}Q(\xi^{-1},\xi\eta)^{-1}P(\xi^{-1},\xi\eta) = P(\xi,\xi^{-1})^{-1}Q(\xi,\xi^{-1})\] 
we obtain using the identity $P(\xi,\xi^{-1})^{-1}Q(\xi,\xi^{-1})  = Q(\xi^{-1},\xi)^{-1}P(\xi^{-1},\xi)$ (cf. (\ref{cip}) )
\[P(\xi,\eta)^{-1}Q(\xi,\eta)P(\xi^{-1},\xi\eta) = Q(\xi^{-1},\xi)^{-1}P(\xi^{-1},\xi)\] 
giving the second equation of (\ref{minussigma}).
It follows that the map 
\begin{equation}\label{Phiext}\Phi:((\xi,\eta),P(\xi,\eta),Q(\xi,\eta))\mapsto (\varphi(\xi,\eta),P(\varphi(\xi,\eta)),Q(\varphi(\xi,\eta))\end{equation}
is an involution and hence the definition (\ref{minussigma}) of the loop cocycle $(P,Q)$ is independent of the choice of the element $(\xi,\eta)$ of an orbit. 
\end{construction}

\begin{corollary} The linear abelian extensions $F(P,Q)$ determined by the conditions  (\ref{sigma}) and (\ref{minussigma}) satisfy the left inverse property. Conversely, any linear abelian extension $F(P,Q)$ of $A$ by $L$ having the left inverse property fulfills the conditions (\ref{sigma}) and (\ref{minussigma}) .
\end{corollary} 
\begin{remark} \label{leftremark} If $L$ is an ordered loop then the definition of the loop cocycle $(P,Q)$ can be simplified: We choose arbitrarily the value $Q(\xi,\eta)$, respectively $P(\xi,\eta)$, if $\xi$ is positive and define the value $Q(\xi^{-1},\xi\eta)$, respectively $P(\xi^{-1},\xi\eta)$, at the element $(\xi^{-1},\xi\eta)$ by 
(\ref{minussigma}). \end{remark}

\subsection*{Right inverse property} 

The opposite loop $F(P,Q)$ of a linear abelian extension having the left inverse property  satisfies the right inverse property, i.e. the identity $(\eta,y)=(\eta,y)  (\xi,x) \cdot (\xi,x)^{-1}$ holds for all $\xi,\eta \in L$ and $x,y \in A$ in $F(P,Q)$. Hence we obtain the following statements: 
\begin{proposition} \label{rightinverse} 
The linear abelian extension $F(P,Q)$ given by (\ref{ext2}) has the right inverse property if and only if $L$ has the right inverse property and the following identities are satisfied: 
\begin{equation} \label{eq4} P(\xi\eta,\eta^{-1})=P(\xi,\eta)^{-1}, \quad Q(\xi\eta,\eta^{-1})=P(\xi,\eta)^{-1}Q(\xi,\eta)P(\eta,\eta^{-1})^{-1}Q(\eta,\eta^{-1}). \end{equation} 
\end{proposition}
\begin{construction}\label{conrip}  Let the loop cocycle $(P,Q)$ be defined by  
\begin{equation}\label{sigmar} \begin{split} &P(\xi,\epsilon):= \hbox{Id},\; Q(\epsilon,\xi):= \hbox{Id},\\ &\xi\mapsto Q(\xi,\epsilon):L\setminus\{\epsilon\}\to\hbox{Aut}(A) \; \hbox{is an arbitrary map},  \\
&P(\xi^{-1},\xi) = P(\epsilon,\xi)^{-1} := p(\xi),\;\hbox{where}\; p: L\setminus\{\epsilon\}\to\hbox{Aut}(A)\;  \hbox{is satisfying (\ref{cip})},\\ &Q(\xi^{-1},\xi) := q(\xi),\;\hbox{where}\; q:L\to\hbox{Aut}(A)\; \hbox{is an arbitrary map}\end{split}\end{equation}
on the subset $\Sigma$.  The permutation $\psi:(\xi,\eta)\mapsto (\xi\eta,\eta^{-1})$ acting on $(L \times L)\setminus\Sigma$ interchanges the disjoint elements of the orbits of the group $\Gamma_\psi$ generated by $\psi$. We choose the values $P(\xi,\eta)$ and $Q(\xi,\eta)$ arbitrarily at one of the elements of the orbits and define  
\begin{equation}\label{minussigmar} P(\xi\eta,\eta^{-1}) := P(\xi,\eta)^{-1}, \quad Q(\xi\eta,\eta^{-1}) := P(\xi,\eta)^{-1}Q(\xi,\eta)P(\eta,\eta^{-1})^{-1}Q(\eta,\eta^{-1}).\end{equation} 
Clearly the loop cocycle $(P,Q)$ is independent of the choice of the element $(\xi,\eta)$ of an orbit.  
\end{construction}

\begin{corollary} The linear abelian extensions $F(P,Q)$ determined by the conditions  (\ref{sigmar}) and 
(\ref{minussigmar}) satisfy the right inverse property. Conversely, any linear abelian extension $F(P,Q)$ of $A$ by $L$ having the right inverse property fulfills the conditions  (\ref{sigmar}) and (\ref{minussigmar}). 
\end{corollary} 
\begin{remark} If $L$ is an ordered loop, then the loop cocycle $(P,Q)$ can be determined by choosing the value for $Q(\xi,\eta)$, respectively $P(\xi,\eta)$ arbitrarily, if $\eta$ is positive, and defining $P(\xi\eta,\eta^{-1})$ and $Q(\xi\eta,\eta^{-1})$ by (\ref{minussigmar}).  \end{remark}

\section{Inverse property} 

\begin{definition}
	A linear abelian extension $F(P,Q)$ of the abelian group $A$ by the loop $L$ is called \emph{strongly linear abelian} if the multiplication satisfies 
	\begin{equation} \label{stli} (\epsilon,a) \cdot (\beta,b) = (\beta,b) \cdot (\epsilon,a) = (\beta,a + b) \end{equation} 
	for any $a,b \in A$ and $\beta\in L$. 
\end{definition}
The loop cocycle $(P, Q)$ determines a strongly linear abelian extension $F(P,Q)$ if and only if $P(\xi,\epsilon)=P(\epsilon,\xi)=\hbox{Id} =Q(\xi,\epsilon)=Q(\epsilon,\xi)$ for every $\xi \in L$. 
\begin{proposition} 
	A strongly linear abelian extension $F(P,Q)$ has the inverse property if and only if $L$ has the  inverse property and the equations   
	\begin{equation} \label{equip}\begin{split} P(\xi\eta,\eta^{-1})=P(\xi,\eta)^{-1}, \quad Q(\xi\eta,\eta^{-1})=P(\xi,\eta)^{-1}Q(\xi,\eta),\\  Q(\xi^{-1},\xi\eta) = Q(\xi,\eta)^{-1}, \quad P(\xi^{-1},\xi\eta) = Q(\xi,\eta)^{-1}P(\xi,\eta) \end{split} \end{equation} 
	hold for all $\xi,\eta \in L$.
\end{proposition}
Now, we build strongly linear abelian extensions $F(P,Q)$ satisfying the inverse property. 
\begin{construction}\label{conip}  Assume that the loop cocycle $(P, Q)$ satisfies $P(\xi,\eta):= \hbox{Id}$, $Q(\xi,\eta):= \hbox{Id}$ for any $\xi,\eta\in \Sigma = \{(\xi,\epsilon); \;\xi\in L\}\cup\{(\epsilon,\xi); \;\xi\in L\}\cup\{(\xi^{-1},\xi); \;\xi\in L\}$. 
	The permutations $\varphi:(\xi,\eta)\mapsto (\xi^{-1},\xi\eta)$ and $\psi:(\xi,\eta)\mapsto (\xi\eta,\eta^{-1})$ acting on the set $(L \times L)\setminus\Sigma$ interchange the pairs of different elements of the orbits of the group $\Gamma_\varphi$ generated by $\varphi$, respectively of the group $\Gamma_\psi$ generated by $\psi$. Let $\Gamma$ be the group generated by $\varphi:(\xi,\eta)\mapsto (\xi^{-1},\xi\eta)$ and $\psi:(\xi,\eta)\mapsto (\xi\eta,\eta^{-1})$. The orbit of the group $\Gamma$ consists of the elements 
		\begin{equation} \label{6orbit}(\xi,\eta),\, (\xi^{-1},\xi\eta), \, (\xi\eta,\eta^{-1}), \,(\eta^{-1},\xi^{-1}),\, ((\xi\eta)^{-1},\xi), \, (\eta, (\xi\eta)^{-1}).	\end{equation} 
	If $(\xi,\eta)\in (L \times L)\setminus\Sigma$, then the permutations  $\varphi:(\xi,\eta)\mapsto (\xi^{-1},\xi\eta)$, \,$\psi:(\xi,\eta)\mapsto (\xi\eta,\eta^{-1})$ and $\theta =: \varphi\cdot\psi\cdot\varphi:(\xi,\eta)\mapsto (\eta^{-1},\xi^{-1})$ are involutions, which are equal if and only if $\xi = \eta$ and $\xi^3 = \epsilon$, otherwise they are pairwise different. The even permutations $\varphi\cdot\psi:(\xi,\eta)\mapsto ((\xi\eta)^{-1},\xi)$ and $\psi\cdot\varphi: (\xi,\eta)\mapsto (\eta, (\xi\eta)^{-1})$ coincide if and only if $\xi = \eta$ and 
	$\xi^3 = \epsilon$, in this case $\varphi\cdot\psi = \psi\cdot\varphi$ is the identity permutation. \\
	In the following we assume that the loop $L$ does not contain elements of order $3$. It follows that the group 
	$\Gamma$ is isomorphic to the permutation group $S_3$ and acts simply transitively on its orbits in 
	$(L \times L)\setminus\Sigma$.\\
	Define the action of $\varphi, \psi\in\Gamma$ on $\hbox{Aut}(A)\times\hbox{Aut}(A)$ by 
	\[\varphi(\mathcal{P,Q}) = (\mathcal{Q}^{-1}\mathcal{P},\mathcal{Q}^{-1}),\; \psi(\mathcal{P,Q}) = (\mathcal{P}^{-1},\mathcal{P}^{-1}\mathcal{Q}),\quad \mathcal{P}, \mathcal{Q}\in\hbox{Aut}(A).\]
	The actions of $\varphi, \psi\in\Gamma$ are involutive and of $(\varphi\cdot\psi)^3$ is the identity map on $\hbox{Aut}(A)\times\hbox{Aut}(A)$. Hence we obtain an action of the group $\Gamma$ on $\hbox{Aut}(A)\times\hbox{Aut}(A)$ as follows:
	\begin{equation}\label{list}\begin{array}{|c |l |l |}
	\hline
	\iota & (\xi,\eta)\mapsto (\xi,\eta)& (\mathcal{P},\mathcal{Q})\mapsto (\mathcal{P},\mathcal{Q}) \\
	\hline
	\varphi & (\xi,\eta)\mapsto (\xi^{-1},\xi\eta)& (\mathcal{P},\mathcal{Q})\mapsto (\mathcal{Q}^{-1}\mathcal{P},\mathcal{Q}^{-1})  \\ 
	\hline
	\psi & (\xi,\eta)\mapsto (\xi\eta,\eta^{-1}) & (\mathcal{P},\mathcal{Q})\mapsto (\mathcal{P}^{-1},\mathcal{P}^{-1}\mathcal{Q}) \\ 
	\hline
	\varphi\cdot\psi\cdot\varphi & (\xi,\eta)\mapsto (\eta^{-1},\xi^{-1})& (\mathcal{P},\mathcal{Q})\mapsto (\mathcal{Q},\mathcal{P})  \\
	\hline
	\psi\cdot\varphi & (\xi,\eta)\mapsto (\eta, (\xi\eta)^{-1}) & (\mathcal{P},\mathcal{Q})\mapsto (\mathcal{P}^{-1}\mathcal{Q},\mathcal{P}^{-1}) \\
	\hline
	\varphi\cdot\psi & (\xi,\eta)\mapsto ((\xi\eta)^{-1},\xi)& (\mathcal{P},\mathcal{Q})\mapsto (\mathcal{Q}^{-1},\mathcal{Q}^{-1}\mathcal{P})  \\
	\hline
	\end{array}\ .\end{equation}
Then the necessary and sufficient condition (\ref{equip}) of the inverse property of the loop $F(P,Q)$ yields the following 
\begin{lemma}
A strongly linear abelian extension $F(P,Q)$ has the inverse property if and only if the action of the group $\Gamma$ on $L\times L$ and $\hbox{Aut}(A)\times\hbox{Aut}(A)$ is equivariant, which means 
	\[\tau Q(\xi,\eta) = Q(\tau(\xi,\eta)),\quad \tau P(\xi,\eta) = P(\tau(\xi,\eta))\quad\hbox{for each}\quad\tau\in\Gamma.\]
\end{lemma}
	We finish now our construction: we choose the values $Q(\xi,\eta)$, respectively $P(\xi,\eta)$, arbitrarily at one of the elements of the orbits of $\Gamma$ in $(L \times L)\setminus\Sigma$  and define the value $Q(\tau(\xi,\eta))$, respectively $P(\tau(\xi,\eta))$, at other elements  $\tau(\xi,\eta)\in (L \times L)\setminus\Sigma$, $\tau\in \Gamma$, corresponding to the  commuting diagrams
	\[\begin{array}{ccc}
	{L\times L} & \xrightarrow{\;\;\mathlarger{\tau}\;\;} & L\times L \\
	\bigg\downarrow{(P,Q)} &  & \bigg\downarrow{(P,Q)} \\
	{\hbox{Aut}(A)} & \xrightarrow{\;\;\mathlarger{\tau}\;\;} & \hbox{Aut}(A)\\
	\end{array}.\] 
\end{construction}
\begin{corollary} \label{slaip} Assume that the loop $L$ does not contain elements of order $3$. The strongly linear abelian extensions $F(P,Q)$ determined by Construction \ref{conip} satisfy the inverse property. Conversely, any strongly linear abelian extension $F(P,Q)$ of $A$ by $L$ having the inverse property can be obtained by Construction \ref{conip}.
\end{corollary} 
\begin{remark} If $L$ is an ordered loop, then we call a pair in $L\times L$ to be positive, if both of its elements are positive. Clearly, in the list (\ref{6orbit}) of elements of an orbit there is precisely one positive pair. The loop cocycle $(P,Q)$ can be determined by choosing the value for $Q(\xi,\eta)$, respectively $P(\xi,\eta)$ freely for positive $(\xi,\eta)$, and defining $P(\xi\eta,\eta^{-1})$ and $Q(\xi\eta,\eta^{-1})$ by the equivariant action of $\Gamma$ on $L\times L$ and $\hbox{Aut}(A)\times\hbox{Aut}(A)$. \end{remark}

\section{Existence of strongly linear abelian extensions with inverse property} 

Let $L$ be a finite loop with cardinality $|L| = l$ without elements of order $3$. If there is a strongly linear abelian extension of an abelian group $A$ by $L$, then it follows from Construction \ref{conip} and Corollary \ref{slaip} that the cardinality of $(L \times L)\setminus\Sigma$ is $l^2 - 3l + 2$ and the action of $\Gamma$ on $(L \times L)\setminus\Sigma$ gives a partition of $\Gamma$ on $(L \times L)\setminus\Sigma$ on the orbits with $6$ elements. Hence $l^2 - 3l + 2$ is divisible by $6$, i.e. there is a natural number $k \ge 1$ such that $l^2 - 3l + 2 = 6k$. Solving the equation we get the expression for $|L| = l$:
\begin{equation}\label{number}l = \frac12(3 + \sqrt{1 + 24k}) = \frac12(3 + h),\quad\hbox{with some}\; k\in\mathbb N,\;\; h^2 = 1 + 24k.\end{equation}
If for a loop $L$ with inverse property there exists a strongly linear abelian extension of an abelian group by $L$, then the cardinality $l = |L|$ satisfies the relation (\ref{number}). Moreover $24k = h^2 - 1 = (h - 1)(h + 1)$ is divisible by $24$ and $h$ is not divisible by $2$ and $3$. 
Conversely, if the cardinality $l$ of a loop $L$ with inverse property satisfies (\ref{number}) then Construction \ref{conip} gives strongly linear abelian extensions of abelian groups by the loop $L$ with $|L| = l$ elements. We obtain
\begin{theorem} Let $L$ be a finite loop, without elements of order $3$, satisfying the inverse property. There exists a loop cocycle on $L\times L$ such that the determined  strongly linear abelian extension has the inverse property if and only if $|L| = l$ satisfies the condition (\ref{number}).
\end{theorem}
In the following list we give all triples $(k,h,l)$ up to $l = 16$ satisfying the condition (\ref{number}) : 

\[\begin{array}{c l l l l l l l l l l l}
k: & 0, & 1, & 2, & 5, &  7,  &  12, &  15, &  22, &  26, &  35, &\dots\\
h: & 1, & 5, & 7, & 11, & 13, &  17, &  19, &  23, &  25, &  29, &\dots\\
l: & 2, & 4, & 5, & 7,  &  8, &  10, &  11, &  13, &  14, &  16, &\dots\\ 
\end{array}\]

\bigskip
\noindent
Author's addresses: \\
\'Agota Figula, Institute of Mathematics, University of Debrecen, H-4002 Debrecen, P.O.Box 400, Hungary. {\it E-mail}: {\tt {}figula@science.unideb.hu} \\[1ex]
P\'eter T. Nagy, Institute of Applied Mathematics, \'Obuda University, 1034 Budapest, B\'ecsi \'ut 96/B, Hungary. {\it E-mail}: {\tt {}nagy.peter@nik.uni-obuda.hu}

\end{document}